\documentclass{article}
\usepackage[UTF8]{ctex}
\usepackage{graphicx}
\usepackage{subcaption}
\usepackage{amsmath, amssymb, amsthm}
\usepackage{mathtools}
\usepackage{bm}
\usepackage{xcolor}
\usepackage{geometry}
\usepackage{tikz}
\usetikzlibrary{calc}
\usepackage{mathrsfs}
\usepackage{textcomp}
\newtheorem{theorem}{Theorem}
\newtheorem{corollary}[theorem]{Corollary}
\newtheorem{lemma}[theorem]{Lemma}
\newtheorem{proposition}[theorem]{Proposition}

\newtheorem{construction}[theorem]{Construction}
\newtheorem{problem}[theorem]{Problem}

\newcommand{\bj}{\bm{j}}

\title{Tight lower bound for the spectral radius of connected graphs \\ with given matching number}
\author{Xinmin Hou$^{1,2}$ \quad and \quad Li Tan$^1$\\
\small $^{1}$School of Mathematical Sciences,\\
\small University of Science and Technology of China, Hefei 230026, Anhui, China\\
\small $^{2}$ Hefei National Laboratory,\\
\small University of Science and Technology of China, Hefei 230088, Anhui, China\\
\small Email: \texttt{xmhou@ustc.edu.cn (X. Hou)},  \texttt{pb22010411@mail.ustc.edu.cn (L. Tan)}}

\date{}

\begin{document}

\maketitle

\begin{abstract}
Let $\mathscr{G}_{n,k}$ denote the family of all connected graphs of order $n$ with matching number $k$. Liu, Lou, and Trevisan~(Linear
Algebra Appl., 2026) posed the following problem: 
Determine the  spectrally minimal graphs in $\mathscr{G}_{n,k}$.  In this paper we prove that for every graph $G \in \mathscr{G}_{n,k}$, 
$
  \rho(G) \ge \sqrt{\frac{n + 2k - 3}{k}},
$
and we completely characterize the extremal graphs when $k \mid (n-3)$. 
As applications, we establish $\rho(G) + k \ge 3\sqrt[3]{n/4}$ for $k \ge 2$, settling the asymptotic order of $\rho + k$ as $\Theta(n^{1/3})$ --- strictly smaller than the $\Theta(\sqrt{n})$ order suggested by the disproved Aouchiche--Hansen conjecture. 
\end{abstract}

\vspace{2mm}

\section{Introduction}

The \emph{spectral radius} of a graph $G$, denoted $\rho(G)$, is the largest eigenvalue of its adjacency matrix $A(G)$. Spectral extremal graph theory, initiated by the classical Brualdi--Solheid problem, seeks to determine how graph parameters constrain the spectral radius and to characterize the graphs that attain the extremal values. Among the many parameters studied, the \emph{matching number} $\mu(G)$ --- the size of a maximum matching --- is of particular interest because of its connections to Hamiltonicity, factor-critical structures, and the total $\pi$-electron energy in chemical graph theory.

For the \emph{upper bound}, Feng, Yu and Zhang~\cite{Feng2007} gave a complete answer. Let $\mathscr{G}_{n,k}$ be the family of connected graphs of order $n$ with matching number $k$. Then:
\begin{itemize}
    \item If $n = 2k$ or $2k+1$, then $\rho(G) \le n-1$, with equality iff $G \cong K_n$.
    \item If $2k+2 \le n < 3k+2$, then $\rho(G) \le 2k$, with equality iff $G \cong K_{2k+1} \cup \overline{K}_{n-2k-1}$.
    \item If $n = 3k+2$, then $\rho(G) \le 2k$, with equality iff $G \cong K_k \vee \overline{K}_{n-k}$ or $G \cong K_{2k+1} \cup \overline{K}_{n-2k-1}$.
    \item If $n > 3k+2$, then $\rho(G) \le \tfrac{1}{2}\bigl(k-1 + \sqrt{(k-1)^2 + 4k(n-k)}\bigr)$, with equality iff $G \cong K_k \vee \overline{K}_{n-k}$.
\end{itemize}


These results on spectral radius upper bounds under constraints of matching number or forbidden paths were later extended to the $A_\alpha$-spectral radius by Yuan and Shao~\cite{Yuan2022}, and to graphs with both bounded clique number and matching number by Wang, Hou and Ma~\cite{Wang2023}. Early work by Hou and Li~\cite{HouLi2002} provided upper bounds on the spectral radius of trees with fixed size and matching number, and Zhai, Xue, and Liu~\cite{ZhaiXueLiu2022} determined maximal $Q$-spectral radii for graphs with fixed size and matching number (see also~\cite{Chang2003,Cioaba2005,LinGuo2007}).

The \emph{lower bound} problem --- determining the minimum spectral radius in $\mathscr{G}_{n,k}$ --- has received comparatively little attention. Sun, Sun, Guo and Tan~\cite{Sun2017} studied the minimal spectral radius of trees for the special case $k \mid (n-1)$ with matching number $k\le 4$.Collatz~\cite{Collatz1957} observed that the star $K_{1,n-1}$ (matching number $1$) has spectral radius $\sqrt{n-1}$, while the path $P_n$ (matching number $\lfloor n/2 \rfloor$) has spectral radius approximately $2$. This suggests that graphs achieving the minimum spectral radius for a given $k$ should have a tree-like structure, and indeed the work of Li and Chang~\cite{Li2014} on studying the extremal graphs for the minimum of the Laplacian spectral radius when the matching number is fixed indicates that extremal graphs typically exhibit ``broom'' or ``caterpillar'' shapes. This leads us to consider that the extremal graphs that minimize the spectral radius of the adjacency matrix among graphs with a given matching number also have similar properties.

The most closely related recent work is that of Liu, Lou and Trevisan~\cite{Liu2026}, who initiated a systematic study of the minimum spectral radius in $\mathscr{G}_{n,k}$.  They posed the following core problem on this subject. 
\begin{problem}[\cite{Liu2026}]\label{P1}
      For the family \(\mathcal{G}_{n,k}\) of graphs of order \(n\) and matching number $k$, what are the spectrally minimal graphs? 
\end{problem}
For $k\le 4$, the authors in \cite{Liu2026} classified all extremal constructions. However, their paper does not provide a general lower bound formula valid for arbitrary $k$, nor does it characterize the extremal structure for large $k$. 

In this article, we present a sharp lower bound valid for all connected graphs with given matching number. To show the lower bound is tight, we construct an explicit family of extremal graphs.

\begin{construction}[Extremal family $\mathcal{T}_{n,k}$]\label{prop:extremal}
Let $k,n$ be positive integers with $k \mid (n-3)$ and $n \ge 2k$. Let $c := (n+2k-3)/k$.
For a tree $T_1$ on a vertex set $X = \{x_1,\dots,x_k\}$ satisfying $2d_{T_1}(x_i) \le c$ for all $i$, construct a tree in the family $\mathcal{T}_{n,k}$ as follows:
\begin{enumerate}
    \item Subdivide every edge of $T_1$ by inserting a new vertex into each edge. For $e = x_i x_j \in E(T_1)$, let $y_e$ denote the inserted vertex and set $Y = \{y_e : e \in E(T_1)\}$, so $|Y| = k-1$. The resulting graph has no edges within $X$; each $y_e \in Y$ is adjacent exactly to $x_i$ and $x_j$.
    \item Attach $f_i$ pendant leaves to each $x_i \in X$, where the non-negative integers $f_1,\dots,f_k$ are defined by
    \[
      2d_{T_1}(x_i) + f_i = c \quad \text{for all } i=1,\dots,k.
    \]
    (These are consistent since $\sum_{i=1}^k f_i = kc - 4(k-1) = n - 2k + 1$.)  
\end{enumerate}
Note that the condition $2d_{T_1}(x_i) \le c$ guarantees $f_i \ge 0$. For small $n$, this condition may fail for certain trees $T_1$ (for example, a star requires $c \ge 2k-2$, which is stronger than $n \ge 2k$). Actually, when $k\ge 4$, for every pair $(n,k)$ with $n\ge 2k$ and $k\mid(n-3)$, there exists at least one tree $T_1$ attaining the extremal value; for instance, when $n=2k+3$, one can take $T_1$ to be a path.
\end{construction}

\begin{figure}[htbp]
    \centering
    \begin{subfigure}{0.48\textwidth}
        \centering
        \includegraphics[width=\linewidth]{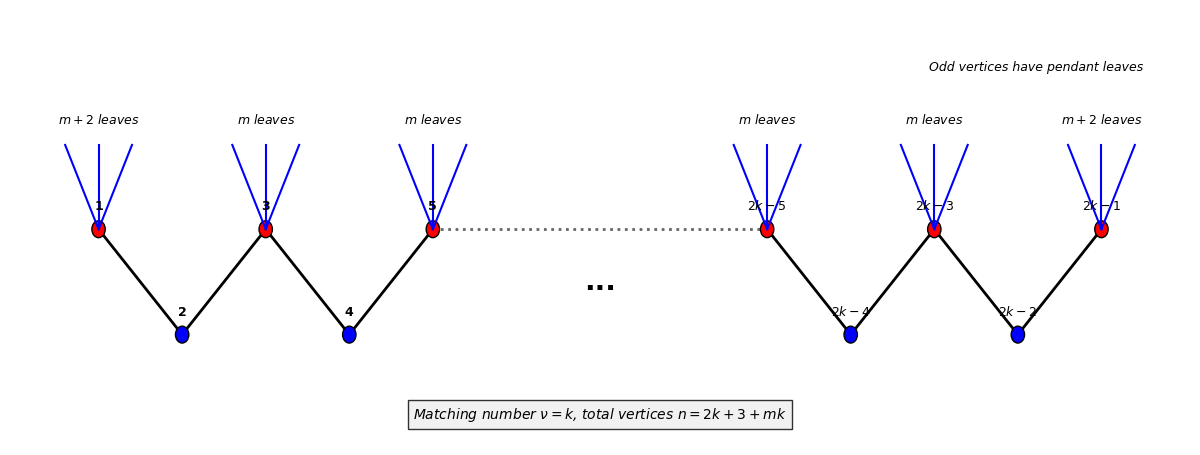}
        \caption{$T_{n,k}$ for $T_1$ being a path.}
        \label{fig:path}
    \end{subfigure}%
    \hfill
    \begin{subfigure}{0.48\textwidth}
        \centering
        \includegraphics[width=\linewidth]{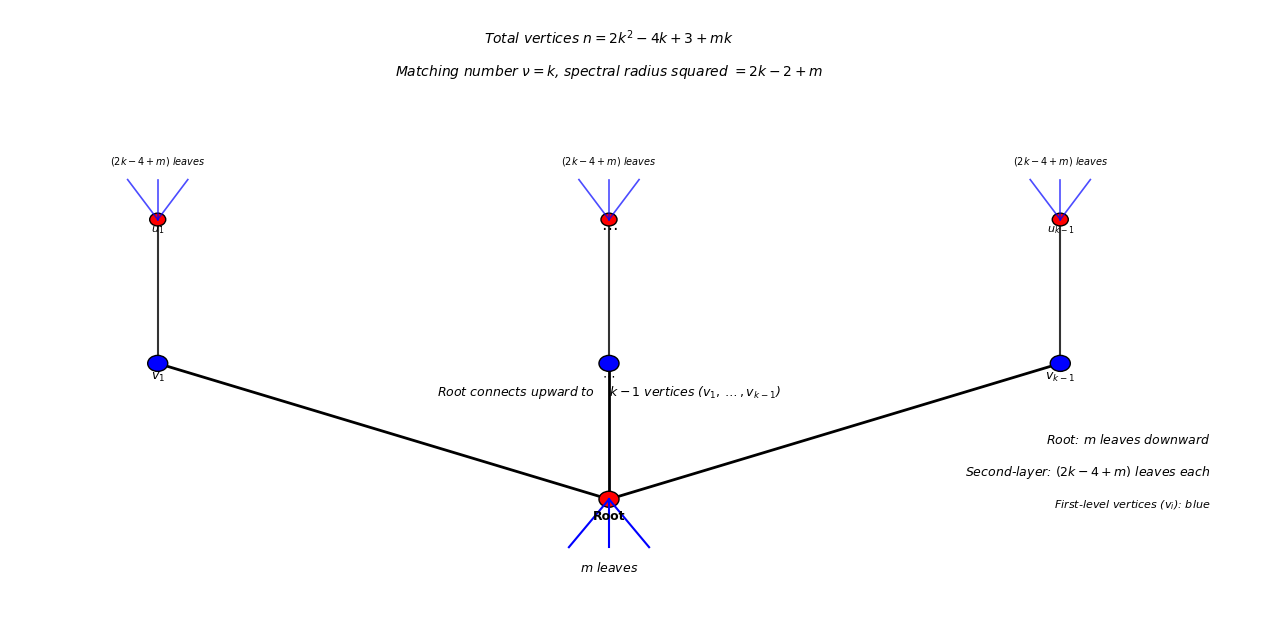}
        \caption{$T_{n,k}$ for $T_1$ being a star.}
        \label{fig:star}
    \end{subfigure}
\caption{$T_{n,k}$ for $T_1$ being a path or a star.}    
\end{figure}


\begin{figure}[htbp]
    \centering
    \includegraphics[width=0.6\linewidth]{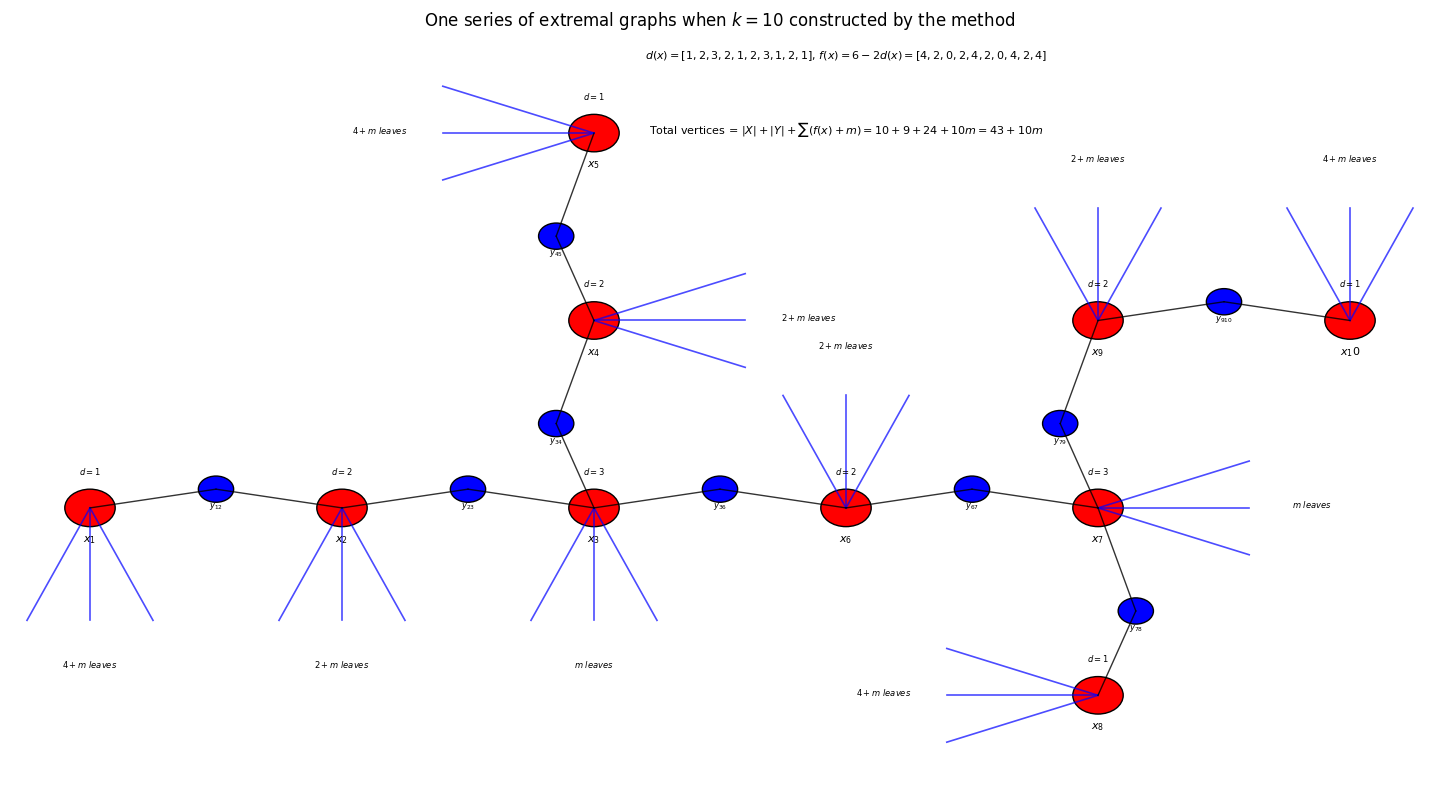}
    \caption{An extremal graph for $k=10$ with a non-path, non-star tree $T_1$.}
    \label{fig:general}
\end{figure}

\begin{theorem}\label{thm:main}
Let $G$ be a connected graph of order $n$ with matching number $k \ge 1$. Then
\begin{equation}\label{eq:main}
  \rho(G) \ge \sqrt{\frac{n + 2k - 3}{k}}.
\end{equation}
The equality holds if and only if $G\in\mathcal{T}_{n,k}$.
\end{theorem}

The following proposition shows the tightness of the lower bound in Theorem~\ref{thm:main}.
\begin{proposition}\label{prop:extremal}
Let \(k,n\) be positive integers such that \(k\mid (n-3)\) and \(n\ge 2k\). Then for any graph  \(G\in\mathcal{T}_{n,k}\),  we have
\[
\rho(G)=\sqrt{\frac{n+2k-3}{k}}.
\]
\end{proposition}

\begin{proof}
Let \(L\) be the set of leaves of $G$. Then the vertex set of $G$ has a partition  \(X,\ Y,\ L\) .  Hence the adjacency matrix is
\[
A(G)=\begin{pmatrix}
0 & B & C\\
B^{\mathsf T} & 0 & 0\\
C^{\mathsf T} & 0 & 0
\end{pmatrix}.
\]
Let \(\lambda=\sqrt{c}\) and define a vector \(\mathbf{v}\) by
\[
\mathbf{v}(x)=1\;(x\in X),\qquad \mathbf{v}(y)=\frac{2}{\lambda}\;(y\in Y),\qquad \mathbf{v}(\ell)=\frac{1}{\lambda}\;(\ell\in L).
\]
Compute \(A(G)\mathbf{v}\):
\begin{itemize}
\item For a leaf \(\ell\in L\) (adjacent to a unique \(x_i\in X\)):
\[
(A(G)\mathbf{v})_\ell = \mathbf{v}(x_i)=1 = \lambda\cdot\frac{1}{\lambda}=\lambda \mathbf{v}(\ell).
\]
\item For \(y\in Y\) (adjacent to two vertices \(x_i,x_j\in X\)):
\[
(A(G)\mathbf{v})_y = \mathbf{v}(x_i)+\mathbf{v}(x_j)=1+1=2 = \lambda\cdot\frac{2}{\lambda}=\lambda \mathbf{v}(y).
\]
\item For \(x_i\in X\): its neighbours consist of \(d_{T_1}(x_i)\) vertices in \(Y\) and \(f_i\) leaves, hence
\[
(A(G)\mathbf{v})_{x_i}=d_{T_1}(x_i)\cdot\frac{2}{\lambda}+f_i\cdot\frac{1}{\lambda}
= \frac{2d_{T_1}(x_i)+f_i}{\lambda}=\frac{c}{\lambda}=\lambda = \lambda \mathbf{v}(x_i).
\]
\end{itemize}
Thus \(A(G)\mathbf{v}=\lambda\mathbf{v}\), i.e., \(\lambda\) is an eigenvalue of \(G\).

Since \(G\) is a connected tree, its adjacency matrix is nonnegative and irreducible. All components of \(\mathbf{v}\) are positive (because \(\lambda>0\)), so by the Perron‑Frobenius theorem the eigenvalue corresponding to a positive eigenvector must be the spectral radius (the largest eigenvalue). Hence \(\rho(G)=\lambda=\sqrt{(n+2k-3)/k}\). This shows that the construction attains the lower bound exactly.

By the way, it is easy to see that the matching number of the constructed graph is $k$: since the graph is a tree, hence bipartite, the matching number equals the minimum vertex cover size, which is at most $|X|=k$. Considering the spanning subgraph induced by $X$ and $Y$, because $n > 2k-1$, there exists some $x$ with attached leaves. 
Take a leaf $w$ adjacent to $x_w\in X$. Then $X\cup Y\cup\{w\}$ induces a tree $T_{2k}$ on $2k$ vertices. Since a tree has at least two leaves, $T_{2k}$ contains a leaf $x_{2k}\in X$. Let $y_{2k}$ be the support vertex of $x_{2k}$. Then $y_{2k}\in Y$. Note that $d_{T_{2k}}(y_{2k})=2$. 
Hence deleting $\{x_{2k},y_{2k}\}$ from $T_{2k}$ does not affect connectivity of the remaining vertices. Let $T_{2k-2}=T_{2k}-\{x_{2k},y_{2k}\}$. The leaves of $T_{2k-2}$ can only be $w$ or elements of the set $X$. Repeatedly applying this operation yields $k$ pairwise disjoint edges, so the matching number is exactly $k$.



\end{proof}



\subsection{Connection to the Aouchiche--Hansen conjecture}


Aouchiche and Hansen~\cite{Aouchiche2010} based on extensive computer experiments, conjectured that for any connected graph $G$ on $n \ge 3$ vertices with matching number $k$,
\[
  \rho(G) + k \ge \sqrt{n-1} + 1.
\]
The first counterexample was found by Stevanovi\`c ~\cite{Stevanovic10}  using  AutoGraphiX and further constructions were given by Wagner~\cite{Wagner2021} using neural-network-guided search. While these disprove the conjecture, they leave open the question of the true asymptotic order of $\rho + k$.

Our lower bound settles this question.

\begin{corollary}\label{cor:rho+k}
For any connected graph $G$ of order $n$ with matching number $k \ge 2$, or  $k=1$ and $n \ge 20$ 
\begin{equation}\label{eq:cor}
  \rho(G) + k \ge 3\sqrt[3]{\frac{n}{4}}.
\end{equation}
Moreover, this bound is tight up to a constant factor: there exist graphs in $\mathscr{G}_{n,k}$ with $\rho + k = \Theta(n^{1/3})$.
\end{corollary}

\begin{proof}
From Theorem~\ref{thm:main} we have $\rho^2 \ge (n+2k-3)/k$. By the AM--GM inequality,
\[
  \rho + k = \frac{\rho}{2} + \frac{\rho}{2} + k \ge 3\sqrt[3]{\frac{\rho^2 k}{4}}
  \ge 3\sqrt[3]{\frac{n+2k-3}{4}} \ge 3\sqrt[3]{\frac{n}{4}}
\]
for $n \ge 2k-1$ (which ensures $n+2k-3 \ge n$).

For $k=1$, we have $\rho = \sqrt{n-1}$, and the inequality $\sqrt{n-1} + 1 \ge 3\sqrt[3]{n/4}$ holds for all $n \ge 20$ by direct verification.

To see that the order $n^{1/3}$ is attained, take  $k = \lfloor \sqrt[3]{\frac{n}{4}} \rfloor$ and apply Proposition~\ref{prop:extremal} (after adjusting $n$ to satisfy the divisibility condition). Then $\rho \approx \sqrt{n/k} \approx \sqrt[3]{2n}$, 
so $\rho + k \approx 3\sqrt[3]{\frac{n}{4}} = \Theta(n^{1/3})$.
\end{proof}

Thus $\rho + k = \Theta(n^{1/3})$, which is strictly smaller than the $\Theta(\sqrt{n})$ order implied by the Aouchiche--Hansen conjecture. This confirms that the star graph $K_{1,n-1}$, while extremal for $k=1$, does not represent the general behavior when $k$ grows with $n$.




The rest of this article is arranged as follows. We present the proof of Theorem~\ref{thm:main} in Section 2. Some remarks are given in the last section.

\section{Proof of the main theorem}

The proof of Theorem~\ref{thm:main} proceeds in three steps. We first reduce to the case of trees, then set up a block matrix argument using a minimum vertex cover, and finally derive the bound via a Schur complement computation.


The key observation is that deleting edges can only decrease the spectral radius, and we can always find a spanning tree that preserves the matching number. {In fact, Liu, Lou and Trevisan~\cite{Liu2026} had proved the lemma as follows:

\begin{lemma}\label{lem:tree}~\cite{Liu2026}
Let $G$ be a connected graph of order $n$ with matching number $k$. Then there exists a spanning tree $T$ of $G$ with matching number $k$. Consequently, $\rho(G) \ge \rho(T)$.
\end{lemma}}




From now on we assume $G$ is a tree. Being a tree, $G$ is bipartite, and by K\"onig's theorem~\cite{Konig1931} the matching number equals the minimum vertex cover size. Let $X$ be a minimum vertex cover of $G$, so $|X| = k$ and $I := V(G) \setminus X$ is an independent set.


The following lemma identifies a small connected ``core'' of $G$.

\begin{lemma}\label{lem:core}
There exists a subset $Y \subseteq I$ with $|Y| \le k-1$ such that the spanning subgraph $G[X \cup Y]$ is connected.
\end{lemma}

\begin{proof}
We build $Y$ greedily. Start with an arbitrary vertex $x \in X$ and let $A_0$ be the set of vertices in $X$ reachable from $x$ via paths that stay entirely inside $X$; set $B_0 = \emptyset$. Clearly $A_0 \cup B_0$ induces a connected subgraph, $|A_0| \ge 1$, and $|B_0| = 0$.

Suppose inductively that we have constructed $A_{w-1} \subseteq X$ and $B_{w-1} \subseteq I$ such that $A_{w-1} \cup B_{w-1}$ induces a connected subgraph, $|A_{w-1}| \ge w$, and $|B_{w-1}| \le w-1$. If $A_{w-1} = X$, we are done. Otherwise, since $G$ is connected and $I$ is independent, there exists a vertex $y \in I \setminus B_{w-1}$ adjacent to some $x' \in A_{w-1}$ and some $z \in X \setminus A_{w-1}$. Set $B_w = B_{w-1} \cup \{y\}$ and let $A_w$ be the set of vertices in $X$ reachable from $A_{w-1} \cup \{z\}$ via paths inside $X$. Then $A_w \cup B_w$ is connected, $|A_w| \ge w+1$, and $|B_w| \le w$.

After at most $k-1$ steps we reach $A_g = X$ for some $g \le k-1$, and we set $Y = B_g$.
\end{proof}

Let $L = I \setminus Y$ be the remaining independent vertices. Every vertex in $L$ has degree $1$ in $G[X \cup Y \cup L]$ (otherwise, if there exists a vertex \(l\in L\) such that \(l\) is adjacent to at least two vertices in \(X\), then these two vertices are connected in the induced subgraph \(G[X\cup Y]\), and they are also connected via \(l\), which would form a cycle, a contradiction.), so each $\ell \in L$ is a pendant leaf adjacent to exactly one vertex in $X$.


Order the vertices of $G$ as $X$, $Y$, then $L$. The adjacency matrix takes the block form
\[
  A(G) = \begin{pmatrix}
    A & B & C \\
    B^{\mathsf T} & 0 & 0 \\
    C^{\mathsf T} & 0 & 0
  \end{pmatrix},
\]
where $A = A(G[X])$ is the adjacency matrix of the induced subgraph on $X$, $B$ records the edges between $X$ and $Y$, and $C$ records the edges between $X$ and $L$. The zero blocks reflect the facts that $Y$ and $L$ are independent sets and there are no edges between $Y$ and $L$.


Assume for a contradiction that $\rho(G) < \lambda$, where $\lambda := \sqrt{(n+2k-3)/k}$. Then $\lambda I - A(G)$ is positive definite. Applying the Schur complement with respect to the block $\mathrm{diag}(\lambda I_Y, \lambda I_L)$, we compute
\begin{align*}
  \begin{pmatrix} I_X & \frac{B}{\lambda} & \frac{C}{\lambda} \\ 0 & I_Y & 0 \\ 0 & 0 & I_L \end{pmatrix}
  \begin{pmatrix} \lambda I_X - A & -B & -C \\ -B^{\mathsf T} & \lambda I_Y & 0 \\ -C^{\mathsf T} & 0 & \lambda I_L \end{pmatrix}
  \begin{pmatrix} I_X & 0 & 0 \\ \frac{B^{\mathsf T}}{\lambda} & I_Y & 0 \\ \frac{C^{\mathsf T}}{\lambda} & 0 & I_L \end{pmatrix} \\
  = \begin{pmatrix} \lambda I_X - A - \frac{1}{\lambda}(BB^{\mathsf T} + CC^{\mathsf T}) & 0 & 0 \\ 0 & \lambda I_Y & 0 \\ 0 & 0 & \lambda I_L \end{pmatrix}.
\end{align*}
Since the left-hand side is positive definite (as a congruence of a positive definite matrix), the top-left block must also be positive definite:
\[
  \lambda I_X - A - \frac{1}{\lambda}(BB^{\mathsf T} + CC^{\mathsf T}) \succ 0.
\]
Let $\bj = (1,1,\dots,1)^{\mathsf T} \in \mathbb{R}^k$. We evaluate the quadratic form:
\begin{equation}\label{eq:quad}
  \bj^{\mathsf T}\!\left(\lambda I_X - A - \frac{1}{\lambda}(BB^{\mathsf T} + CC^{\mathsf T})\right)\!\bj > 0.
\end{equation}

We now compute each term separately. Let $u = e(G[X])$ be the number of edges inside $X$.

\paragraph{The $\bj^{\mathsf T} A \bj$ term.}
This is twice the number of edges in $G[X]$, so $\bj^{\mathsf T} A \bj = 2u$.

\paragraph{The $\bj^{\mathsf T} CC^{\mathsf T} \bj$ term.}
Since each leaf $\ell \in L$ is adjacent to exactly one vertex in $X$, the matrix $CC^{\mathsf T}$ is diagonal: $(CC^{\mathsf T})_{ij} = 0$ for $i \ne j$, and $(CC^{\mathsf T})_{ii}$ is the number of leaves attached to the $i$-th vertex of $X$. Hence $\bj^{\mathsf T} CC^{\mathsf T} \bj = \operatorname{tr}(CC^{\mathsf T}) = \sum_{\ell \in L} 1 = |L|$.

\paragraph{The $\bj^{\mathsf T} BB^{\mathsf T} \bj$ term.}
For $i \ne j$, $(BB^{\mathsf T})_{ij}$ counts the number of common neighbors in $Y$ of the $i$-th and $j$-th vertices of $X$. Since $G$ is a tree, any two vertices of $X$ share at most one common neighbor in $Y$ (otherwise a $4$-cycle would exist). Define an auxiliary graph $\widetilde{G}$ on vertex set $X$ by declaring $x_i \sim x_j$ in $\widetilde{G}$ if either $x_i \sim x_j$ in $G[X]$ or $x_i,x_j$ have a common neighbor in $Y$. Then for $i \ne j$,
\[
  (BB^{\mathsf T})_{ij} = A(\widetilde{G})_{ij} - A_{ij},
\]
because $x_i \sim x_j$ in $G[X]$ and $x_i,x_j$ sharing a neighbor in $Y$ cannot both occur (that would create a $3$-cycle). On the diagonal, $(BB^{\mathsf T})_{ii}$ is the degree of the $i$-th vertex of $X$ in $G[X,Y]$.

Let $D$ be the diagonal part of $BB^{\mathsf T}$. Then $BB^{\mathsf T} = A(\widetilde{G}) - A + D$, and
\[
  \bj^{\mathsf T} BB^{\mathsf T} \bj = \bj^{\mathsf T} A(\widetilde{G}) \bj - \bj^{\mathsf T} A \bj + \operatorname{tr}(D).
\]


Since $\operatorname{tr}(D) = \operatorname{tr}(BB^{\mathsf T}) = |E(X,Y)|$ (each edge between $X$ and $Y$ contributes exactly $1$ to the diagonal of $BB^{\mathsf T}$), we obtain
\begin{equation}\label{eq:BB}
  \bj^{\mathsf T} BB^{\mathsf T} \bj = \bj^{\mathsf T} A(\widetilde{G}) \bj - 2u + |E(X,Y)|.
\end{equation}

\paragraph{Assembling the inequality.}
Substituting all terms into \eqref{eq:quad}:
\begin{align*}
   \lambda k - 2u - \frac{1}{\lambda}\bigl(\bj^{\mathsf T} A(\widetilde{G}) \bj - 2u + |E(X,Y)| + |L|\bigr)>0.
\end{align*}
Multiplying by $\lambda$ and rearranging:
\begin{equation}\label{eq:assemble}
  \lambda^2 k > \lambda \cdot 2u + \bj^{\mathsf T} A(\widetilde{G}) \bj - 2u + |E(X,Y)| + |L|.
\end{equation}
We have $|E(X,Y)| + |L| = e(G) - u = n - 1 -u$. Since $|Y| \le k-1$ by Lemma~\ref{lem:core}, and $\widetilde{G}$ is connected (because $G[X \cup Y]$ is connected), we have $\bj^{\mathsf T} A(\widetilde{G}) \bj \ge 2(k-1)$. Also $\lambda \ge 1$ since $n \ge 2k-1$. Hence the right-hand side of \eqref{eq:assemble} satisfies
\begin{align*}
  \lambda \cdot 2u + \bj^{\mathsf T} A(\widetilde{G}) \bj - 2u + n - 1 -u
  &\ge u(2\lambda - 3) + 2(k-1) + n - 1\\
  &\ge u(2\lambda - 3) + n + 2k -3.
\end{align*}
Thus from \eqref{eq:assemble},
\[
  \lambda^2 k > u(2\lambda - 3) + n + 2k -3.
\]
Substituting $\lambda^2 = (n+2k-3)/k$, the inequality becomes
\[
  u(2\lambda - 3) < 0,
\]
If $u = 0$, then $0 > 0$, contradiction. Therefore, $u > 0$. Consequently,  we have
\[
  2\lambda-3 <0.
\]
When $k \geq 2$,  $T$ must contain a path of length $3$.
Since 
\[
  \rho(\text{path of length } r) = 2\cos\left(\frac{\pi}{r+2}\right),
\]
we have 
\[
  \lambda \geq \rho(G) \geq \rho(\text{path of length }3) = \frac{\sqrt{5}+1}{2} > \frac{3}{2}, \text{ a contradiction.}
\]
When $k = 1$,  the graph is a star $K_{1,n-1}$, $\rho(G) = \sqrt{n-1}$,  a contradiction again.

 In summary, by contradiction we obtain
\[
  \rho(G) \geq \sqrt{\frac{n + 2k - 3}{k}}
\]


\paragraph{Characterization of equality.}

Assume equality holds in \eqref{eq:main}. Tracing back through the proof, all inequalities must become equalities:
\begin{enumerate}
    \item[(1)] $u = 0$, so $X$ is an independent set.
    \item[(2)] $|Y| = k-1$, so $\widetilde{G}$ is a tree.
    \item[(3)] $\bj$ is a Perron eigenvector of $BB^{\mathsf T} + CC^{\mathsf T}$, meaning that for every $x \in X$, the quantity $2d_{G[X,Y]}(x) + d_{G[X,L]}(x)$ is constant.
    \item[(4)] No vertex in $Y$ is adjacent to three or more vertices of $X$ (otherwise $\widetilde{G}$ would contain a triangle, contradicting that it is a tree). Hence every vertex in $Y$ is adjacent to exactly two vertices of $X$.
\end{enumerate}
Conditions (1)--(4) precisely describe the construction of Proposition~\ref{prop:extremal}: $\widetilde{G}$ is a tree $T_1$ on $X$, each edge of $T_1$ is subdivided by a vertex of $Y$, and pendant leaves are attached to vertices of $X$ to equalize the weighted degree. The equality condition $k \mid (n-3)$ follows from the requirement that $(n+2k-3)/k$ be an integer.

Conversely, every graph constructed in Proposition~\ref{prop:extremal} satisfies these conditions and achieves equality. This gives a complete characterization.

\section{Conclusion}
We have established a sharp lower bound $\rho(G) \ge \sqrt{(n+2k-3)/k}$ for the spectral radius of connected graphs with matching number $k$, valid for all $n \ge 2k-1$. The bound is tight whenever $k \mid (n-3)$, and we have characterized all extremal graphs: they are obtained by subdividing the edges of an arbitrary tree on $k$ vertices and then attaching pendant leaves to equalize degrees. 

Our result answers Problem~\ref{P1} for all $n \ge 2k-1$ with $k \mid (n-3)$. For general $n$ and $k$, we have the following corollary from Theorem~\ref{thm:main}.

\begin{corollary}\label{thm:sharp}
Let $G \in \mathscr{G}_{n,k}$ and let $n-3\equiv r\pmod k$. Then
\[
  \min_{G \in \mathscr{G}_{n,k}} \rho(G) < \sqrt{\frac{n + 3k - 3 - r}{k}}.
\]

\end{corollary}
\begin{proof}
Note that $k \mid (n + k - r - 3)$.  Then the proof is obtained by deleting $k - r$ pendant leaves from an extremal graph of order $n + k - r$.    
\end{proof}

Based on  our computational experiments, we present a general extremal construction for general $n$ and $k$.
\begin{construction}[General extremal family $\mathcal{T'}_{n,k}$]\label{const:general}
Let $k$ and $n$ be positive integers with $n \ge 2k$. Define 
\[
c := \left\lfloor \frac{n+2k-3}{k} \right\rfloor.
\]
Let $T_1$ be a tree on vertex set $X=\{x_1,\dots,x_k\}$ for which there exist nonnegative integers $f_1,\dots,f_k$ satisfying
\[
2d_{T_1}(x_i)+f_i \in \{c,\, c+1\} \quad \text{for all } i=1,\dots,k.
\]
We construct a family of trees $\mathcal{T'}_{n,k}$ as follows:
\begin{enumerate}
    \item Subdivide every edge of $T_1$ by inserting a new vertex into each edge. 
    \item Attach $f_i$ pendant leaves to each $x_i\in X$, where the $f_i$'s are chosen as above.
\end{enumerate}
The resulting graph is a tree of order $n$ with matching number $k$ (as in the proof of Construction 2). 
\end{construction}
Our computational experiments suggest that for general $n$ and $k$, all spectrally minimal  graphs in $\mathcal{G}_{n,k}$ are contained in $\mathcal{T'}_{n,k}$.
This leads to the following question: for arbitrary $n$ and $k$, must every spectrally minimal graph in $\mathcal{G}_{n,k}$ belongs to $\mathcal{T'}_{n,k}$. 

\vspace{5pt}
\noindent{\bf Acknowledgements}:
This work was supported by the National Key Research and Development Program of China (2023YFA1010203), the National Natural Science Foundation of China (12471336), and Quantum Science and Technology-National Science and Technology Major Project (2021ZD0302902).

\end{document}